\documentclass[a4paper,10pt]{amsart}
\usepackage[latin1,latin9]{inputenc}  
\usepackage[T1]{fontenc} 
\usepackage{verbatim}
\usepackage{amsmath}
\usepackage{amssymb}
\usepackage{amscd}
\usepackage{german}

\newcommand{\N}{\mathbb{N}}
\newcommand{\Z}{\mathbb{Z}}

\newtheorem{theorem}{Theorem}[section]

\title{Otto Toeplitz: Algebraiker der unendlichen Matrizen } 
\thanks{Erweiterte Fassung eines
Vortrags zum 80j\"ahrigen Bestehen der Mathematischen Semesterberichte, der am 23. April 2013 in Bonn gehalten wurde.}

\author{Stefan M\"uller-Stach}
\email{mueller-stach@uni-mainz.de}
\address{Johannes Gutenberg-Universit\"at Mainz}

\begin{document}

\begin{abstract} Otto Toeplitz ist ein Mathematiker, dessen Schicksal exemplarisch f\"ur die Vernichtung der j\"udischen
wissenschaftlichen Elite in Deutschland durch die Nationalsozialisten ist.
Sein Einfluss in der Mathematik ist noch heute, besonders durch den Begriff der Toeplitzmatrizen, deutlich sp\"urbar. 
Die kulturgeschichtliche Bedeutung von Toeplitz ist ebenso gro{\ss}. 
Als Gr\"undungsherausgeber von zwei Zeitschriften hat er sein Engagement f\"ur die 
Didaktik und die Wissenschaftsgeschichte untermauert. 
Wir geben einen Einblick in sein Schicksal und seine Gedankenwelt unter dem Einfluss von David Hilbert
und Felix Klein.
\end{abstract} 

\maketitle

\section*{Einleitung}

Die Mathematischen Semesterberichte wurden 1932 von Heinrich Behnke und Otto Toeplitz gegr\"undet.
Heinrich Behnke mit seinen herausragenden Verdiensten um die Mathematikdidaktik, die Lehrerfortbildung,
seinem mathematischen \OE{}uvre und seiner in der Nachkriegszeit so einflussreichen Schule ist 
im Hinblick auf die Semesterberichte schon gew\"urdigt worden \cite{hartmann1,hartmann2}. 
Mit dem vorliegenden Aufsatz wollen wir nun stattdessen einmal Otto Toeplitz gerecht werden. 

Es soll insbesondere herausgearbeitet werden, wie nahe die Idee der Semesterberichte zur Gedankenwelt von Toeplitz 
und der von Behnke ist. 
Toeplitz, der aus Breslau stammte, wurde zun\"achst in G\"ottingen von David Hilbert und Felix Klein stark beeinflusst. 
Ab den 1920er Jahren entwickelte er aber vollkommen eigene Ideen, sowohl in seiner Forschung als auch im Bereich
der Wissenschaftsgeschichte und der Hochschuldidaktik. Die beiden von ihm gegr\"undeten Zeitschriften, seine historischen Seminare sowie
seine Reden und Schriften zeigen seine Individualit\"at und seine mittlerweile erreichte inhaltliche Distanz zu seinen beiden Vorbildern, mit denen er 
auch eine umfangreiche Korrespondenz unterhielt.

F\"ur Toeplitz bedeutete Deutschland die Erf\"ullung seiner wissenschaftlichen Tr\"aume. 
Durch die Nationalsozialisten wurde Toeplitz aus seinem Amt entfernt und emigrierte 1939 nach 
Jerusalem, wo er Anfang 1940 starb. Zwischen 1933 und 1939 hatte er weiter intensiv 
seine Ideen verfolgt und sich f\"ur die Juden in Deutschland in mehrfacher Weise eingesetzt.

\section*{Otto Toeplitz: Der Mensch}

Otto Toeplitz muss eine bemerkenswerte Pers\"onlichkeit mit herausragenden Charaktereigenschaften 
gewesen sein. Dies wird wunderbar in der liebevollen Beschreibung aus einem Nachruf von Heinrich Behnke \cite[S. 1]{behnke2} ausgedr\"uckt:  \\

\noindent
{\it Otto Toeplitz war Mathematiker. Aber das charakterisiert ihn bei weitem nicht. Er war zugleich P\"adagoge und Historiker seines Faches und zu allererst ein hilfsbereiter, 
mitf\"uhlender Mensch. Er geh\"orte zu den Charakteren, deren Bild man nicht wieder aus dem Ged\"achtnis verliert, wenn man mit Ihnen auch nur einmal ein 
Gespr\"ach gef\"uhrt hat. Er lauschte den Worten seines Gegen\"ubers, wer auch immer es war, mit ganz ungeteilter Aufmerksamkeit - wie es so selten geschieht - und sprach dann
selbst mit gro{\ss}em psychologischen Einf\"uhlungsverm\"ogen. Seine freundlichen Augen, sein Mienenspiel, all seine Bewegungen standen im Banne der geistigen Vorg\"ange in ihm. 
Er sprach mit einer wohltuenden W\"arme, einer Offenheit und einer Kritik an sich selbst, die immer wieder \"uberraschte ... \\ 
Aus allem, was er tat, sprach sein 
Verantwortungsgef\"uhl gegen\"uber den geistigen M\"achten, seine Liebe zu Deutschland, vor allem zur akademischen Jugend, und seine Verpflichtung gegen\"uber der deutschen Universit\"at, 
- nicht zuf\"allig hat Karl Jaspers die 1. Auflage seines Buches ,,Die Idee der Universit\"at'' ihm gewidmet.}\\

\noindent Die darin beschriebene Liebe zu Deutschland, und das Engagement von Otto Toeplitz als Universit\"atsprofessor wurden schmerzlich verraten durch das Schicksal, 
das ihm nach 1933 zusammen mit den anderen j\"udischen Kollegen widerfuhr \cite[S. 103]{ausstellung}. 
Weitere Nachrufe und Beschreibungen der Person von Otto Toeplitz verfassten Stefan Hildebrandt \cite{bms319} und Wilhelm Klingenberg \cite{bms143}.
Nachrufe zum wissenschaftlichen Werk erschienen von Peter Lax in \cite{bms319}, Jean Dieudonn\'e in \cite{gohberg}, Israel Gohberg in \cite{gohberg} 
und Gottfried K\"othe in \cite{gohberg} sowie \cite{koethe}. 
Die Kinder von Otto und Erna Toeplitz, die Tochter Eva Wohl \cite{wohl} und der Sohn Dr. Uri Toeplitz \cite{uri}, haben Autobiographien geschrieben, in denen 
ihre Eltern eine wichtige Rolle spielen. Eva Wohl best\"atigt die Einsch\"atzung von Behnke in ihren Erinnerungen 
,,So einfach liegen die Dinge nicht'' (eine Redensart von Otto Toeplitz) \cite[S. 40]{wohl}: \\

\noindent {\it Um seine Vorlesungen vorzubereiten, machte Vater oft einen Spaziergang durch die Kaiser-Friedrich-Stra{\ss}e bis ans Rheinufer. Manchmal durfte ich ihn dabei begleiten.
Dann hielt er mich an seiner Hand und lie{\ss} mich hopsen. Dabei pr\"ufte er mich auf mathematische Scherzfragen, von denen eine war: Zwei V\"ater und zwei S\"ohne a{\ss}en zusammen
drei Eier, wie ist das m\"oglich? Er hatte ein ganzes Buch solcher Fragen, sie stammten aus dem Besitz seines Vaters. Ein anderer Scherz war die Gleichung 
$({\rm Klim}+{\rm Bim})^2={\rm Klim}^2+2 {\rm Klimbim} + {\rm Bim}^2$ ... Mein Vater war ein sehr attraktiver Mann. Es muss wohl sein au{\ss}erordentlicher Charme 
gewesen sein, der ihn so beliebt bei Frauen machte. Er konnte sich leicht in die Lage seines Gegen\"ubers versetzen und war deshalb interessant in der Unterhaltung. 
Eine seiner Maximen war: ,,Man muss hinh\"oren k\"onnen''. Zudem war er \"au{\ss}erst gebildet und interessiert an vielen Dingen, auch solchen, die au{\ss}erhalb der Mathematik lagen.
Wie ich schon erw\"ahnt habe, liebte er Musik, und auch die leiblichen Gen\"usse, gutes Essen und Trinken, waren ihm sehr wichtig. Dazu war er einer der g\"utigsten Menschen, die man sich 
vorstellen kann.}\\

\noindent Das Buch von Eva Wohl berichtet ausf\"uhrlich \"uber die Kindheit in Deutschland und ihr Leben nach ihrer Auswanderung im Jahre 1939.
Dr. Erich (Uri) Toeplitz war Fl\"otist und Mitbegr\"under des Israel Philharmonic Orchestra, das seit 1936 bestand. Er schreibt \"uber seinen Vater in \cite[S. 49]{uri}: \\

\noindent
{\it Von Zeit zu Zeit kamen Mathematiker nach Kiel zu Besuch. Sie wurden dann gastlich aufgenommen und wohnten auch bei uns. H\"aufiger Gast war Ernst Hellinger, Professor
in Frankfurt am Main, ein alter Freund und Studienfreund meines Vaters aus Breslau ... Hellinger spendierte zum Nachmittagskaffee regelm\"a{\ss}ig die sch\"onsten St\"ucke Kuchen, 
die wir Kinder aussuchen und einkaufen durften ...
Eines Tages kamen in Kiel zwei ,,Wanderv\"ogel'' zu uns, beide in kurzen Hosen, Professoren der Mathematik, welche die Ostsee entlang gewandert waren. Es waren Reidemeister aus
K\"onigsberg und Rademacher aus Breslau. Mit dem letzteren entwickelte sich eine n\"ahere Beziehung. Beide, Vater und er, hielten an ihren Universit\"aten Vorlesungen ,,f\"ur H\"orer
aller Fakult\"aten'', in denen sie versuchten, die Mathematik f\"ur Laien verst\"andlich darzustellen ... Rademacher schlug vor, die besten dieser Vorlesungen auszusuchen, 
und in Buchform zu ver\"offentlichen. Vater stimmte zu, obwohl er dies lieber alleine gemacht h\"atte, und ,,Von Zahlen und Figuren, Proben mathematischen Denkens f\"ur Liebhaber
der Mathematik'', angefangen in Kiel, erschien im Jahr 1930 ... Schon wenn man das Vorwort liest, erkennt man den pers\"onlichen Stempel meines Vaters, sagen wir, seinen wunderbar 
musikalischen Stil. W\"ahrend der Arbeit kam es h\"aufig vor, da{\ss} sich die Autoren nicht einig waren, ob ein Kapitel wirklich allgemeinverst\"andlich sei. Da wurde ich eingeschaltet.}\\

\noindent Uri Toeplitz starb 2006. Auch sein Sohn Gideon war Musiker und leitete viele Jahre lang als Manager das Pittsburg Symphony Orchestra. 
Otto Toeplitz selbst hat seine Musikalit\"at auch in seinen mathematischen Werken ausgedr\"uckt, wie Uri Toeplitz und auch Eva Wohl beide betonen. \\

\noindent Max Born, der sp\"atere Physik-Nobelpreistr\"ager, schrieb 1940 einen Nachruf auf Otto Toeplitz in Nature \cite{born}. 
Die algebraische Seite von Toeplitz in der Tradition von Hilbert bringt er darin deutlich zum Ausdruck: \\

\noindent {\it Professor Otto Toeplitz, who died on March 15 in Jerusalem, was one of that rare type of scholars who combine with a high standard of specialized research
a deep interest in the history of their science and its bearing on general questions ...
Toeplitz's mathematical interest was wide and covered all branches of research, but was deeply rooted in algebra. He liked to consider
analysis as an algebra of an infinite number of variables, an example of which was given by Hilbert's treatment of integral equations as special cases of linear equations
and quadratic forms of infinite variables. } \\

\noindent Toeplitz lebte auch fort in seinen zahlreichen Sch\"ulern: Charlotte Schroeckh (Kiel 1917), August Pr\"uss (Kiel 1918), Ulrich John (Kiel 1922), 
Hans Hansen (Kiel 1923), Robert Schmidt (Kiel 1923), Anna Stein (Kiel 1927), Walter Stein (Kiel 1930), Helmut Ulm (Bonn 1933), Anton Weber (Bonn 1934),
Marie-Luise Schluckebier (Bonn 1935), Hans Schwerdtfeger (Bonn 1935), Dora Reimann (Bonn 1935), Arthur Donald Steele (Bonn 1936), 
Elisabeth Hagemann (Bonn 1937) und Bruno Ritzdorff (Bonn 1938). Diese Liste mit bemerkenswert vielen Frauen ist vermutlich unvollst\"andig. 
Sie wurde mit Hilfe des Doktorandenarchivs von Renate Tobies bei der DMV (http://dmv.mathematik.de) und der Sch\"ulerkartothek 
von Toeplitz und Bessel-Hagen \cite[S. 403]{neuenschwander} erstellt.

\section*{Breslau, Kiel und Bonn: Biographisches}

Toeplitz wurde am 1. August 1881 in Breslau geboren und wuchs dort auf. 
Sein Vater Emil, und sein Gro{\ss}vater Julius waren beide Mathematiklehrer. 
Seine Mutter Pauline (eine Cousine von Emil) starb an Tuberkulose als Toeplitz 11 Jahre alt war. 
Er wuchs in strenger Erziehung als Einzelkind im Haushalt des Vaters Emil auf, der am gleichen Johannes Gymnasium in Breslau Lehrer war, auf das Toeplitz ging. \\

\noindent Toeplitz studierte dann auch in Breslau Mathematik. Sein Doktorvater war Jakob Rosanes und Toeplitz promovierte 
1905 mit der Dissertation  ,,\"Uber Systeme von Formen, deren Funktionaldeterminante identisch verschwindet'' im Gebiet der algebraischen Geometrie.  
Lebenslange Freundschaften verbanden Toeplitz aus der Breslauer Zeit mit Max Born, Richard Courant und Ernst Hellinger. \\

\noindent 1906 wechselte Toeplitz nach G\"ottingen. Dort wurde er Sch\"uler und Mitarbeiter von David Hilbert, der zu dieser Zeit \"uber Integralgleichungen arbeitete, einem 
seiner Schwerpunkte nach den Arbeiten zur Invariantentheorie, zu den Grundlagen der Geometrie und zur Zahlentheorie. 
Otto Toeplitz setzte Hilbert ein Denkmal mit seinem Aufsatz ,,Der Algebraiker Hilbert'' \cite{toeplitz7}.
In diesem Artikel setzt sich Toeplitz mit dem Werk Hilberts auseinander und stellt interessanterweise seine eigene Forschung in eine historische Tradition mit Hilberts algebraischen Arbeiten. 
Dar\"uber berichten wir weiter unten. Das Verh\"altnis von Hilbert und Toeplitz war sehr freundschaftlich, die Familien kannten sich und es gab auch in den Folgejahren noch 
Briefwechsel \cite{toeplitz11,toeplitz15}. 
Franz Hilbert, der Sohn von David und K\"athe Hilbert, erhielt offenbar Nachhilfe von Toeplitz (ebenso \"ubrigens von Courant \cite[S. 20]{reid1}), wie aus einer Postkarte aus dem Nachlass 
\cite{toeplitz11} hervorgeht. \\

\noindent Ein enges Verh\"altnis existierte auch zwischen Felix Klein und Otto Toeplitz. Dabei drehte es sich unter anderem um Fragen der Hochschuldidaktik. Klein besch\"aftigte sich schon seit 
Ende des 19. Jahrhunderts mit didaktischen Fragen und der Reform der Schulen \cite{MU}. Kleins Stil und Ideenwelt mit ihren konkreten Objekten unterschied sich deutlich 
von der abstrakteren Welt von Hilbert. Hierzu ein Zitat von Behnke \cite[S. 3]{behnke2}; \\

\noindent {\it Aber Toeplitz' Wirken ist ganz und gar nicht auf diese Forschung beschr\"ankt. Dazu war er ein zu vielseitiger Geist, und in G\"ottingen hat ihn nicht nur 
Hilbert beeinflu{\ss}t. Felix Klein war ihm auch ein eindrucksvolles Vorbild und hat ihn veranla{\ss}t, sich auch fr\"uhzeitig mit Fragen des Unterrichts zu befassen. Klein scheint dem 
jungen G\"ottinger Dozenten schon fr\"uh viel zugetraut zu haben. Er sprach mit ihm viel \"uber die Didaktik der Anf\"angervorlesung, und wir k\"onnen uns lebhaft vorstellen, 
wie Toeplitz den Geist der Hilbertschen Schule den klassisch-anschaulichen Auffassungen von Klein (der ja nicht einmal den Begriff der abstrakten Gruppe zulassen wollte) gegen\"uber
verteidigte.} \\

\noindent In der G\"ottinger Zeit war Toeplitz auch in Kontakt mit Hermann Minkowski \cite[S. 1]{behnke2} und Hermann Weyl \cite[S. 25]{reid1}, die Distanz zu beiden war aber 
gr\"o{\ss}er als zu Hilbert und Klein. 
Im Jahr 1913 wechselte Toeplitz als au{\ss}erordentlicher Professor nach Kiel. Aufgrund von Verz\"ogerungen wurde er erst einige Jahre sp\"ater, wie versprochen, Ordinarius
als Nachfolger von Leo Pochhammer. Zeitweise war Toeplitz Dekan in Kiel und ma{\ss}geblich verantwortlich f\"ur einige Berufungen. In seiner Kieler Zeit waren auch 
Helmut Hasse als Privatdozent sowie Julius Stenzel, Heinrich Scholz und Ernst Steinitz in Kiel. Ein Versuch, Karl Jaspers aus Heidelberg nach Kiel zu holen, scheiterte. 
Anscheinend hatte es dabei eine antisemitische Reaktion innerhalb der Fakult\"at 
gegeben, die auf die j\"udische Frau von Jaspers zielte, siehe den Beitrag von Uri Toeplitz in \cite[S. 28]{bms143}. Jaspers und Toeplitz f\"uhrten einen Briefwechsel, 
der heute in Marbach untergebracht ist, ebenso existiert ein umfangreicher Briefwechsel zwischen Heinrich Behnke und Karl Jaspers \cite{hartmann3}. 
Jaspers widmete Toeplitz die erste Auflage seines Buches ,,Die Idee der Universit\"at''. Toeplitz arbeitete mit Ernst Hellinger mehrere Jahre an 
dem langen \"Ubersichtsband \cite{toeplitz3} \"uber Integralgleichungen, der in der Enzyklop\"adie der mathematischen Wissenschaften erschien.
\"Uber das Verh\"altnis zu Ernst Hellinger schreibt Uri Toeplitz \cite[S. 49]{uri}: \\

\noindent {\it ,,Onkel Henga'', wie wir Kinder ihn nannten, war Junggeselle und besonders nett zu uns Kindern. Wir st\"orten aber z.B., wenn die Erwachsenen sich bei Tisch
unterhalten wollten. So setzte er uns ein Schweigegeld aus, und das klappte ... Hellinger spendierte zum Nachmittagskaffee regelm\"a{\ss}ig die sch\"onsten St\"ucke Kuchen, die wir
Kinder aussuchen und einkaufen durften; das gab es bei uns nur, wenn er da war. Sp\"ater, zwischen 1934 und 1936, als ich in Frankfurt lebte, hat er mich regelm\"a{\ss}ig zum 
Mittagessen eingeladen.}  \\

\noindent Ernst Hellinger veranstaltete zusammen mit Max Dehn, Carl Ludwig Siegel und Paul Epstein ein historisches Seminar in Frankfurt \cite{siegel}. 
Toeplitz hatte durch ihn somit Beziehungen zu diesem Frankfurter Seminar, wie sie auch bei Behnke \cite[S. 4]{behnke2} erw\"ahnt werden. Hellinger  
verbrachte einige Zeit in einem Konzentrationslager, bis er in die USA auswandern konnte. Er bekam 1939 eine Stelle als Lecturer 
an der Northwestern University in Evanston bei Chicago und starb dort 1950 als emeritierter Professor. \\

\noindent In dieser Kieler Zeit begann auch Toeplitz Einsatz f\"ur die Lehrerausbildung \cite{toeplitz8}. 
Dazu richtete er ein mathematisch-didaktisches Kolloquium ein \cite[S. 170]{hartmann1} und k\"ummerte sich sowohl um die Fortbildung der Lehrer 
aus dem Umland und als auch um die Ausbildung der Referendare. Bei diesen Aktivit\"aten ist nat\"urlich der starke Einfluss von Felix Klein 
zu erw\"ahnen, der die sogenannte Kleinsche Reform des Unterrichts und die Ferienkurse f\"ur Lehrer begr\"undet hatte. 
Einige didaktische Einsichten aus diesen Veranstaltungen publizierte Toeplitz in den ersten Ausgaben der Semesterberichte, in denen noch 
Unterrichtskonzepte f\"ur Lehrer behandelt wurden. So besprach er dort zum Beispiel ,,eingekleidete Gleichungen'' oder die ,,Erneuerung'' 
der Schulgeometrie \cite[S. 60-62]{bms319}. \\

\noindent In Kiel hatte Toeplitz auch intensive Kontakte zu Heinrich Scholz und Julius Stenzel, siehe \cite[S. 4]{behnke2}. Zusammen wurden Seminare 
zur Mathematik der Griechen abgehalten. Mehrere Jahre sp\"ater in der Bonner Zeit gr\"undete Toeplitz zusammen mit Julius Stenzel und Otto Neugebauer die Zeitschrift 
,,Quellen und Studien zur Geschichte der Mathematik, Astronomie und Physik'' \cite{quellen}. K\"othe \cite[S. 14]{bms143} berichtet \"uber Wolfgang Franz, 
der das Kieler Seminar besuchte, und dort eine Diskussion \"uber das fehlende Konzept von Irrationalzahlen bei den Griechen miterlebte. \\

\noindent Im Jahr 1927 erging an Toeplitz der Ruf nach Bonn. Als Nachfolger von Eduard Study trat er die Stelle im Sommersemester 1928 an. 
Toeplitz und seine Familie erlebten einen Aufschwung durch das Ankommen in Bonn, Uri Toeplitz nennt den Ruf sogar eine ,,Erl\"osung'' nach den 
Kieler Jahren \cite[S. 28]{bms143}, die f\"ur die Familie wohl durch Provinzialit\"at und schlechte Erlebnisse, wie der Episode mit Jaspers Frau, gepr\"agt waren, 
und betont auch an anderer Stelle \cite[S. 54]{uri}, dass ihm die rheinische Lebensart mehr entgegenkam. \\

\noindent Im Jahr 1930 erschien das Buch ,,Von Zahlen und Figuren'' \cite{toeplitz2} von Hans Rademacher und Otto Toeplitz, das in der Kieler Zeit begonnen wurde. Es ist 
in viele Sprachen \"ubersetzt worden. Max Born schreibt dar\"uber in \cite{born}:\\

\noindent {\it But Toeplitz's general attitude to mathematics, which he preferred to consider more as an art than as a science, is more clearly visible
in his quasi-popular book ''Von Zahlen und Figuren'' written with Rademacher in 1930, which seems to me a masterpiece of that class of scientific literature
which attempts to instruct a wider public in the fundamental ideas of science. It is not easy, but certainly fascinating reading.} \\

\noindent Heinrich Behnke bemerkt zu diesem Buch in \cite[S. 6]{behnke2}:\\

\noindent {\it Jede der 30 Nummern dieses Buches ist in ihrer Abgeschlossenheit und Abgewogenheit ein Kunstwerk f\"ur sich. Alles war in Vortr\"agen 
vor einem gr\"o{\ss}eren Publikum auf seine Eignung durchgeprobt. So ist dieses kleine B\"uchlein viel mehr als es scheint, und bildet eine 
gl\"anzende Einf\"uhrung in heutige mathematische Probleme f\"ur solche Leser, die, ohne die Zeit aufzubringen, umfangreiche mathematische Theorien zu erlernen, 
in das Wesen der Mathematik eindringen wollen. Alle anderen popul\"aren Werke der letzten Jahrzehnte \"ubertrifft es weit an p\"adagogischem Geschick und an 
sicherer Zielsetzung. Die ganze Darstellung wirkt wie ein Ausschnitt aus klassischen Werken.} \\

\noindent 
Wie in \cite[S. 22]{bms319} berichtet wird, zeigte Felix Hausdorff bereits im Vorfeld der Berufung gro{\ss}es Interesse an Toeplitz und 
es gab einen sehr freundschaftlichen Kontakt zur Familie Toeplitz. Uri Toeplitz schreibt in \cite[S. 26]{bms143}: \\

\noindent {\it Unsere Familien waren befreundet. Er war sehr vielseitig und hat, unter anderem, ein viel aufgef\"uhrtes Lustspiel geschrieben. Er war auch sehr musikalisch, 
und ich durfte mit ihm vierh\"andig spielen. } \\

\noindent Israel Gohberg berichtet in \cite[S. 34]{bms143} \"uber die mathematische Zusammenarbeit von Hausdorff und Toeplitz \"uber das Spektrum von Operatoren, 
was sich vermutlich auf das sogenannte Toeplitz-Hausdorff Theorem bezieht. \\

\noindent Toeplitz entwickelte in dieser Zeit auch sein Interesse in Mathematikgeschichte in Zusammenarbeit mit Erich Bessel-Hagen und Oskar Becker weiter.  
Mit seinem Sch\"uler Gottfried K\"othe bekam er neue Ideen in der Forschung \cite{toeplitz9}. Die hohen Studentenzahlen in Bonn veranlassten ihn,  
weitere Ideen in der Hochschuldidaktik der Anf\"angervorlesungen zu entwickeln. 
Aus diesen Quellen speiste sich auch seine ,,genetischen Methode'', die in einem Buchmanuskript \cite{toeplitz1} zu einer
Anf\"angervorlesung zur Infinitesimalrechnung umgesetzt wurde, welches posthum Gottfried K\"othe herausgab.  
In den folgenden Abschnitten berichten wir ausf\"uhrlicher \"uber die historischen und didaktischen Ans\"atze von Toeplitz sowie seine mathematischen Arbeiten.

\section*{Die Zeit in Deutschland nach 1933 und in Jerusalem}

Ab 1933 wurden Juden mit Verweis auf das ,,Gesetz zur Wiederherstellung des Berufsbeamtentums'' aus dem Beamtendienst entfernt.
Toeplitz konnte noch bis 1935 vor den ,,N\"urnberger Gesetzen'' im Amt bleiben, da er bereits vor dem 1. Weltkrieg
Beamter war, verlor aber die Pr\"ufungsberechtigung. 
Otto Toeplitz dr\"uckte seine Gef\"uhle \"uber die Zeit kurz nach 1933 in einem Brief vom 22. Januar 1934 an David Hilbert aus \cite{toeplitz15}: \\

\noindent{\it ... und da Klein nicht mehr ist, werden Sie der gewichtigste Zeuge von diesem St\"uck Geschichte sein, das Sie durch 35 Jahre miterlebt haben. Es war mir 
bisher verg\"onnt, dass ich meine mathematische Arbeit fortsetzen konnte. Meine Sch\"uler sind mir treu geblieben und arbeiten weiter mit mir mit, 
soweit der Dienst ihnen irgend Zeit dazu l\"asst. Ich glaube, diese jungen Menschen k\"ampfen um den Ausgleich ihrer Pflichten und Ziele nicht minder, als es diese Zeit auch mir 
auferlegt, widerstreitende Empindungen in mir zu balancieren. Fr\"uher habe ich stets versucht, alles, was deutsch und alles, was j\"udisch in mir war, 
hunderprozentig zu bejahen und zu einer fruchtbaren Synthese mit einander zu bringen; und ich habe an der M\"oglichkeit davon nie gezweifelt und sehe auch 
heute keinen Grund, meine Ansicht dar\"uber zu \"andern. Auch heute verleugne ich nichts von dem, was in mir deutsch ist und nichts von dem, was in mir j\"udisch ist.
Aber heute ist statt der Synthese der Kampf der beiden Momente gegen einander getreten, und das ist es, was soviel von den Kr\"aften auffrisst, die man gerne einem
besseren Zweck widmen m\"ochte.}\\

\noindent Uri Toeplitz beschreibt in \cite[S. 28]{bms143} auch diese Jahre. Nach dem erfolgreichen Ruf und dem Eintreffen in Bonn vergingen nur 
5 Jahre, in denen Toeplitz angemessen arbeiten konnte. Zitieren wir daraus: \\

\noindent{\it Es kam das Jahr 1933. Als ,,Altbeamter'' - vor 1914 im Dienst - wurde er nicht gleich entlassen. Das geschah erst mit den ,,N\"urnberger Gesetzen'' des
Jahres 1935. Er wurde pensioniert, nicht einmal emeritiert ... T\"aglich kamen erniedrigende Probleme, wie z.B. der Befehl, von nun an die Vorlesungen mit dem deutschen Gru{\ss} 
zu beginnen ... Allm\"ahlich r\"uckten die Kollegen von ihm ab. Erst besuchte man sich nicht, dann gr\"u{\ss}te man nicht mehr. Er konnte nicht mehr in die Universit\"at, 
in die Bibliothek gehen ... Doch fand mein Vater auch eine, seine positive Antwort auf die Ereignisse und die wachsende Verfolgung. 
Er stellte seine Dienste der Bonner j\"udischen Gemeinde zur Verf\"ugung. Als Sachverst\"andiger in Fragen der Erziehung \"ubernahm er die Aufgabe, die j\"udischen Kinder vor
der nun feindlichen Umgebung zu bewahren und eine j\"udische Schule zu gr\"unden ... F\"ur Jahre war er mit der Schule verbunden. Er wurde auch Hochschuldezernent der 
,,Reichsvertretung der deutschen Juden''. Er unternahm es, begabte Studenten auszuw\"ahlen und auf Stipendien in Ausland zu schicken ... Er schrieb Artikel \"uber j\"udische Erziehung 
und reiste viel herum, besonders nach Berlin.}\\

\noindent Es handelte sich um die private j\"udische Volksschule in der Koblenzer Strasse 32, die 1935-1938 von Hans Herbert Hammerstein geleitet wurde \cite[S. 41]{wohl}.
Erna Toeplitz beschreibt diese Schule ebenfalls in \cite[S. 45]{bms319}: \\

\noindent {\it Schon im August \"uberbrachte Rabbiner Levy meinem Mann die Bitte des Vorstandes, er m\"ochte in einer Gemeinde-Versammlung eine Rede
\"uber die Gr\"undung einer j\"udischen Volksschule in Bonn halten, die von Samuel und Weiss ausging. Mit der ihm eigenen Intensit\"at studierte mein Mann die Fragen der \"ubrigen 
Schulen im Rheinland -- nur in Bonn und D\"usseldorf gab es keine -- die finanziellen Fragen und das Lehrer-Angebot, und am 16.10.33 fand im Gemeindehaus die Gr\"undungsversammlung 
der Schule statt. Sofort wurden 52 Kinder aus Bonn und benachbarten Gemeinden wie Beuel, Wesseling, Hersel etc. angemeldet. }\\

\noindent Es gibt einen Brief von Toeplitz an Helmut Ulm vom 18. November 1935 \cite{uri,bms143}, der ebenso die Zeit 1933-1935 beschreibt: \\

\noindent {\it Die Zeit war f\"ur mich -- abgesehen von einer sehr h\"a{\ss}lichen Episode -- gar nicht so aufregend. Da{\ss} ich dieses Amt, 
wie es jetzt war, nicht mehr weiterzuf\"uhren brauche, ist f\"ur mich nahezu eine Befreiung. Obgleich mir nie was Unangenehmes zugestossen ist in diesen zwei Jahren, 
war diese st\"andige Bewu{\ss}theit, dieses Abhandenkommen jeder Naivit\"at, eine Belastung, die mir immer erst dann ganz bewu{\ss}t geworden ist, wenn ich davon 
einmal befreit war, und so auch jetzt.} \\

\noindent Im August 1935 schrieb Toeplitz an Courant \cite[S. 79]{RSS}:\\

\noindent {\it Denn dies ist meine Auffassung: wir m\"ussen die Stellen, auf denen
man uns l\"a{\ss}t, bis zum letzten Augenblick halten, nicht als ob eine
Besserung in Sicht w\"are -- ausgeschlossen -- sondern weil wir sonst
in irgend einer Form der allgemeinen Judenheit zur Last fallen,
mindestens einem anderen die Stellen wegnehmen. Ich betrachte
es als ein Opfer, da{\ss} ich der Judenheit bringe, auf diesem Posten
auszuhalten.}\\

\noindent Otto Toeplitz war mit einigen Bonner Kollegen aus anderen Wissenschaften befreundet. Da war zum Beispiel der Geologe Hans Cloos und 
seine Familie. Cloos war 1926-1951 Direktor des Geologisch-Pal\"aontologischen Instituts in Bonn. Zuvor hatte er in Marburg und Breslau gearbeitet. 
Cloos war ein dezidierter Gegner der Nationalsozialisten. Er hatte 1933 auf einer Reise nach Afrika den Sohn Walter Toeplitz besucht, der 1933 nach 
S\"udafrika ausgewandert war. Am Tag vor der Reichskristallnacht brachte er Otto Toeplitz nach Aachen in Sicherheit, damit er dem Konzentrationslager entgehen konnte. 
Nach dem Krieg wurde Cloos von den Besatzungsm\"achten f\"ur den Wiederaufbau gebraucht und sogar als B\"urgermeister von Bonn ins Spiel gebracht \cite[S. 133]{wohl}. 
Paul Ernst Kahle, ein Bonner Orientalist, und seine Frau Marie haben ebenfalls eine besondere Rolle als Helfer in der Zeit der Verfolgung gespielt, 
und beide haben dies in Berichten dokumentiert \cite{kahle}. \\

\noindent Ein Artikel mit dem Titel ,,Was macht der j\"udische Abiturient ?'' erschien in der Zeitung des Central-Vereins \cite{toeplitz10}. Darin gibt Toeplitz
Ratschl\"age zur Berufswahl von j\"udischen Abiturienten. Dieser Artikel vom M\"arz 1937 strahlt noch einen gewissen Optimismus aus, wie Toeplitz selbst im letzten Satz betont, jedoch 
deutlich und vielsagend als ,,eingeschr\"ankten Optimismus`` benennt. 
Er konzentrierte sich beim Inhalt auf Jungen, wie die folgende kuriose Passage mit Kommentar der Redaktion unterstreicht: \\

\noindent {\it M\"adchen sollen unter den heutigen Umst\"anden jedes Studium, dessen Ziel eine produktive wissenschaftliche Bet\"atigung ist, unterlassen. (Wir k\"onnen der Meinung 
des Verfassers in der von ihm gew\"ahlten apodiktischen Form nicht ganz beipflichten. Hervorragend begabte M\"adchen sollten auch heute den Wettbewerb mit m\"annlichen 
Bewerbern aufnehmen (Die Schriftleitung).)} \\

\noindent Otto Toeplitz hat in dieser Zeit, wie man sieht, nicht aufgegeben und vermutlich auch bis 1937 mathematisch weiter gearbeitet, wie K\"othe in \cite[S. 22]{bms143} berichtet. 
Im Sommer 1938 besuchte er Pal\"astina auf Einladung von Uri und trug dort unter anderem an der Jerusalemer Universit\"at vor. Von Seiten der Administration 
wurde ihm eine Stelle in Aussicht gestellt \cite[S. 113]{uri}. 
In dieser Zeit muss wohl die Entscheidung gefallen sein, endg\"ultig nach Pal\"astina auszuwandern. Andere Reisepl\"ane in die USA
werden von Richard Courant \cite[S. 213-214]{reid1} und etwas ausf\"uhrlicher von Uri Toeplitz \cite[S. 119]{uri} geschildert, 
Es scheint, dass die Auswanderung nach Jerusalem offenbar auch unsicher war. 
Nach der R\"uckkehr nach Deutschland versch\"arfte sich der Antisemitismus in Deutschland noch mehr. Die Hilfe von Hans Cloos 
w\"ahrend der Reichskristallnacht war wom\"oglich lebensrettend. Uri Toeplitz berichtet in \cite[S. 30]{bms143} 
von einer Vorladung durch die Gestapo im Winter 1938/39, die die Gesundheit des Vaters weiter besch\"adigte.\\

\noindent In einem Brief an Heinrich Behnke vom 1. Januar 1939 beschreibt Toeplitz die Zeit des Wartens vor der Abreise nach Jerusalem \cite{toeplitz11}: \\ 

\noindent {\it Meine Sache liegt z.Z. in den H\"anden der engl. Regierung zu Jerusalem. Es sind bange Wochen des Wartens, bange auch deshalb, weil ich indessen meine hiesigen B\"ucher
abbrechen muss, um nach Eintreffen des Zertifikats nicht zuviel Zeit mit den \"Amtern zu verlieren. Die \"Amter sind bisher alle sachlich und h\"oflich. } \\

\noindent In diesem Brief k\"undigte er auch einen Abschiedsbesuch in M\"unster an, der tats\"achlich Anfang Februar stattfand, wie Gottfried K\"othe in \cite[S. 23]{bms143} 
eindringlich beschreibt: \\

\noindent {\it Wenige Tage vor seiner Ausreise kam Toeplitz noch nach M\"unster, wo Behnke und ich ihn am Bahnhof abholten und mit ihm in ein naheliegendes Caf\'e gingen.
Toeplitz war gekommen, um mit uns, seinen Freunden, Abschied zu nehmen. Ich sehe ihn noch vor mit, aus dem abfahrenden Zug uns zuwinkend, ernst und traurig. Wir ahnten, da{\ss} es
ein Abschied f\"ur immer war. } \\

\noindent Erna und Otto Toeplitz verlie{\ss}en Deutschland im Februar 1939, getrennt \"uber die Schweiz \cite[S. 74-75]{wohl}. Es existieren einige Briefe aus dem Jahr 1939, 
die nochmals die menschlichen Qualit\"aten von Otto Toeplitz beleuchten. So gibt es zwei Briefe von Polya aus dem Mai 1939, in denen sich beide \"uber Schur und Hausdorff 
austauschen. Es ist bekannt, dass Toeplitz genau Buch f\"uhrte (genauer gesagt Karteikarten, deren Verbleib mir nicht bekannt ist) 
\"uber die Schicksale und Selbstmorde von j\"udischen Kollegen. In einem Brief an Max Born schreibt er 1939 \cite[S. 84]{wohl}: \\

\noindent {\it Am 1.2.39 habe ich in Basel die Grenze \"uberschritten und bin seit dem 10.2.39 in der Verwaltung der hiesigen Universit\"at als wissenschaftlicher Berater t\"atig ...
Von diesem Hexenkessel, Deutschland, will ich nicht viel schreiben. Derjenige, der 1933 mitgemacht hat, wie Du, hat noch lange keine Vorstellung von 1938/1939. 
Worte k\"onnen nicht wiedergeben, 
in welche Verzerrung dieses vortreffliche Volk geraten ist.}\\

\noindent Im Rest des Briefes spricht Toeplitz \"uber seine neue T\"atigkeit, die nicht leicht sei, trotz der Unterst\"utzung durch den F\"orderer 
Salman Schocken, den Chef der Verwaltung der Hebr\"aischen Universit\"at zu Jerusalem.
Nach einem Jahr in Jerusalem starb Otto Toeplitz am 15. Februar 1940 mit 58 Jahren in Jerusalem, nachdem seine Gesundheit sich wieder
verschlechtert hatte. Er wurde auf dem \"Olberg beerdigt. Zu seinen Ehren h\"angt seit 1995 \cite{bms319} 
eine Gedenktafel im Mathematischen Institut der Universit\"at Bonn. 


\section*{Der Algebraiker Hilbert}

Otto Toeplitz schrieb 1922 einen Artikel \"uber David Hilbert mit dem Titel ,,Der Algebraiker Hilbert'' \cite{toeplitz7}.
Aus dem ganzen Artikel spricht die Bewunderung von Toeplitz f\"ur Hilbert.
Nachdem er zun\"achst die Arbeiten Hilberts zur Zahlentheorie als den Kern von Hilberts Leistungen w\"urdigt, geht er auf dessen Leistungen in der Algebra ein: \\

\noindent {\it Von dem Algebraiker Hilbert zu erz\"ahlen, so zu erz\"ahlen, dass der Naturwissenschaftler ein Bild erh\"alt, und selbst derjenige einen Eindruck, der
dessen mathematische Exerzitien bei der Gleichung der Ellipse aufgeh\"ort haben, ist vermessen ... Indessen Hilbert nimmt zur Algebra eine besondere Stellung ein.
Seine Methoden sind von einer ganz anderen Art wie die eines Kronecker, Frobenius, Gordan; es sind Methoden, auf die bei ihrem Auftreten keiner der Beteiligten gefasst war ... 
Und dieser Schimmer einer M\"oglichkeit, von dem Algebraiker Hilbert zu erz\"ahlen, soll deshalb nicht ungenutzt bleiben, weil sich in seinen algebraischen Forschungen
ein Wesenszug offenbart, der sich von dem Hintergrunde der Algebra plastischer abhebt als von dem seiner vielen Arbeitsgebiete, und der ein Grundpfeiler ist f\"ur die Analyse
des mathematischen Menschen Hilbert.} \\

\noindent In der Folge geht Toeplitz ausf\"uhrlich auf die Leistungen Hilberts in der Invariantentheorie ein. Im letzten Teil des Artikels schl\"agt er allerdings die 
Br\"ucke zum Gebiet der Integralgleichungen und der unendlichen Matrizen, Toeplitz eigenem Arbeitsgebiet: \\

\noindent {\it Aber Hilbert hat noch einmal eine umfassende algebraische Leistung vollbracht, in der Theorie der Integralgleichungen und unendlichvielen Variablen, deren analytische Seite 
an anderer Stelle dieses Heftes zur Geltung kommt ... Das einigende Band f\"ur alle diesen Theorien ist ein algebraischer Gedanke ... So hat Hilbert in seiner Theorie der
unendlichvielen Variablen ein gro{\ss}es Haus errichtet, in dem viele einzelne Untersuchungen ger\"aumige Wohnungen gefunden haben; ... So steht der Algebraiker Hilbert in
ausgepr\"agter Eigenart vor uns, wie nicht ein jeder Mathematiker. }\\

\noindent Aus dieser Beschreibung liest man die Bewunderung f\"ur Hilbert heraus, es wird aber auch deutlich, wie sehr Toeplitz seine eigene Forschung 
vor einem algebraischen Hintergrund sah.  
Im letzten Abschnitt des Artikels verbindet Toeplitz die beschriebenen Wesensz\"uge von Hilbert mit seiner Exaktheit und dessen Beurteilung der Mathematik in den 
ber\"uhmten Pariser Problemen und der Grundlagenfrage der Mathematik. An dieser Stelle verteidigt Toeplitz den formalen axiomatischen 
Standpunkt Hilbert gegen\"uber Weyl, Brouwer und anderen, die dem Konstruktivismus nahestanden \cite{reid2}. Dies ist schon eine Vorwegnahme von Teilaspekten 
des sp\"ateren Annalenstreits.

\section*{Briefe von Klein und Hilbert: J\"udische Berufungslisten }

J\"udische Mathematikprofessoren in Deutschland spielten erst seit dem 19. Jahrhundert eine Rolle, 
angefangen mit Moritz Abraham Stern, einem Sch\"uler von Gauss und dem ersten j\"udischen Ordinarius. Im Fall von 
Carl G.J. Jacobi hat Felix Klein bemerkenswerte Aussagen in diesem Zusammenhang gemacht. Nachdem er zuvor viele kleine Spitzen \"uber die 
Person von Jacobi in den Text einstreut, und die wissenschaftlichen Leistungen und die Schule Jacobis sehr hoch
bewertet, schreibt er \cite[S. 114]{klein}: \\

\noindent {\it Ehe ich die Betrachtungen \"uber Jacobi schlie{\ss}e, m\"ochte ich noch eine Tatsache erw\"ahnen, die mir unter
dem Gesichtspunkt der Charakteristik dieses Mannes als auch unter dem der Entwicklung unserer Wissenschaft nicht unwichtig erscheint. Bekanntlich
hatte das Jahr 1812 die Emanzipation der Juden in Preu{\ss}en gebracht. Jacobi ist der erste j\"udische Mathematiker, der in Deutschland eine 
f\"uhrende Stellung einnimmt. Auch hiermit steht er an der Spitze einer gro{\ss}en, f\"ur unsere Wissenschaft bedeutungsvollen Entwicklung. 
Es ist mit dieser Ma{\ss}nahme eine neues, gro{\ss}es Reservoir mathematischer Begabung f\"ur unser Land er\"offnet, dessen Kr\"afte neben dem 
durch das franz\"osische Emigrantentum gewonnenen Zuschu{\ss} sich in unserer Wissenschaft sehr bald fruchtbar erweisen. Es scheint mir durch 
solch eine Art Blutserneuerung eine starke Belebung der Wissenschaft gewonnen zu werden. }\\

\noindent David Rowe thematisiert in \cite{rowe} die Berufungsangelegenheiten um die Jahrhundertwende in das
zwanzigste Jahrhundert, und damit zusammenh\"angend die ,,Jewish Question'', besonders anhand der Person Adolf Hurwitz. 
Hurwitz Ruf nach G\"ottingen und die Rollen von Felix Klein, David Hilbert und anderen in dieser Zeit werden beleuchtet. 
F\"ur Toeplitz waren Hilbert und Klein nicht nur mathematische Vorbilder, er hat sie auch in seinen eigenen 
Berufungsaktivit\"aten als Gutachter bem\"uht. Man kann sich denken, dass Klein und Hilbert bei vielen Berufungen in der Mathematik 
entscheidenden Einfluss nehmen konnten. 
David Hilbert schrieb an Toeplitz am 8. Februar 1920 einen Brief, in dem er auf ein Berufungsverfahren in Kiel auf eine 
Professur f\"ur Algebra eingeht \cite{toeplitz11}: \\

\noindent
{\it An erster Stelle w\"urde ich Steinitz vorschlagen, der es verdient und der kommt und mit dem sie ausgezeichnet fahren 
werden (Fussnote: den ich auch rein wissenschaftlich f\"ur den erfolgreichsten Forscher unter den Genannten halte.). 
Wenn Sie einen nicht-j\"udischen Mathematiker an 2ter oder 3ter Stelle nennen wollen, so kommen Bieberbach, f\"ur den ich 
\"ubrigens keineswegs schw\"arme, und Hausdorff (Anmerkung: sic!) in Frage ... Hellinger ist ein Faulpelz und soll erst 
wieder etwas machen, ehe man ihn auf die Liste setzt, ebenso Dehn. 
\"Ubrigens werden alle in n\"achster Zeit Rufe bekommen, da ein gro{\ss}er Mangel an Mathematikern besteht.} \\

\noindent Steinitz hat die Stelle auch bekommen. Sein mathematischer Ruf begr\"undete sich auf seine Arbeiten zur Theorie 
der algebraischen K\"orpererweiterungen \cite{steinitz}. Er hatte sp\"ater auch anerkannte Leistungen in der Polyedertheorie erbracht und 
starb 1928 in Kiel. Erstaunlicherweise hatten Nachkommen der Familien Steinitz und Toeplitz in Jerusalem Kontakt, wie man aus 
Dokumenten im Internet (http://family.steinitz.net) entnehmen kann. 
Felix Klein hat in einem bemerkenswerten Brief an Toeplitz am selben Tag, dem 8. Februar 1920, auch zur selben Sache 
Stellung genommen. Wir geben nur das Ende des Briefes wieder, weil es die Art und Weise deutlich macht, mit der 
Klein agierte \cite{toeplitz11}: 

\medskip 
\noindent
{\it  Nun komme ich, um nichts zur\"uckzuhalten, auf die Frage des Antisemitismus. Sie wissen, wie ich es selbst immer gehalten habe, 
seit ich 1874 die Berufung von Gordan nach Erlangen veranla{\ss}te: mir war der einzelne Jude willkommen, indem ich voraussetzte, da{\ss}
er mit den \"ubrigen Mitgliedern der Universit\"at kooperieren werde. Aber nun haben sich im Laufe der Zeit die Gegens\"atze prinzipiell versch\"arft. 
Wir haben auf der einen Seite nicht nur ein ungeheures, der merkw\"urdigen Leistungsf\"ahigkeit entprechendes Vordr\"angen des Judentums, 
sondern das Hervorkommen der j\"udischen Solidarit\"at (welche dem Stammesgenossen auf alle Weise in erster Linie zu helfen strebt). Dazu nun
als R\"uckwirkung den starken Antisemitismus. Das Problem ist ein allgemeines, bei dem Deutschland, soweit nicht gerade die moderne \"ostliche 
Einwanderung in Betracht kommt, nur eine sekund\"are Rolle spielt. Niemand kann sagen, wie sich das Ding weiterentwickelt. Aber ich mache
darauf aufmerksam, da{\ss} die s\"amtlichen f\"unf Gelehrten, die Sie f\"ur Ihr Ordinariat in Aussicht nehmen, j\"udischen Ursprungs sind. 
Ist das eine zweckm\"a{\ss}ige Politik ? Ich nehme von 
vorneherein an, da{\ss} Sie das nicht beabsichtigt haben. Man kann auch beinahe so argumentieren: da{\ss} der an allen Universit\"aten etc. 
latent vorhandene Antisemitismus die christlichen Kandidaten so bevorzugt habe, da{\ss} nur noch j\"udische zur Verf\"ugung stehen. 
Aber ich bitte doch dar\"uber nachzudenken. Wir treiben m\"oglicherweise in Gegens\"atze hinein, die f\"ur unsere gesamten Zust\"ande unheilvoll werden k\"onnten.} \\

\noindent In unserer heutigen Zeit wirken manche dieser Worte, wohl als Abbild des ,,Zeitgeistes'', klarerweise ,,politisch inkorrekt''. 
Klein spricht den \"uberall vorhandenen Antisemitismus jedoch klar aus. Erstaunlich ist das Vorhersehen von Entwicklungen, besonders im letzten Satz. 
Uri Toeplitz geht auf diesen Brief sowie weitere beteiligte Personen wie Eduard Study und Edmund Landau ebenfalls in \cite[S. 139-141]{uri} ein.

\section*{Heinrich Behnke und die Semesterberichte}

Heinrich Behnke, der in M\"unster seine ganze Wirkung und Bedeutung entfaltete, kam aus Hamburg von Erich Hecke, wo auch Toeplitz regelm\"a{\ss}ig zu Gast war. 
Behnke schreibt in \cite[S. 48/49]{behnke1}, dass Hecke manchmal leicht abf\"allig vom ,,dicken Toeplitz'' sprach. 
Behnke erw\"ahnt darin auch, dass Toeplitz Sinn f\"ur den Unterricht ausgepr\"agter war, 
als das im Kreis der Hilbertsch\"uler \"ublich war. Erich Hecke war Sch\"uler von David Hilbert. 
Besonders interessant in Behnkes Autobiographie  \cite{behnke1} sind die Beschreibungen 
des Universit\"atslebens am Anfang des 20ten Jahrhunderts. Trotz des Titels ,,Semesterberichte'' enth\"alt sie ansonsten leider nur wenige Erinnerungen an Toeplitz. 
Nachdem Toeplitz in Kiel bereits Fortbildungsveranstaltungen f\"ur Mathematiklehrer durchgef\"uhrt hatte \cite{toeplitz8}, seine beachteten Vortr\"age \"uber die Schnittstelle 
Schule--Universit\"at \cite{toeplitz4,toeplitz5} gehalten hatte und auch h\"aufig als Gast der 1919 gegr\"undeten Hamburger Universit\"at mit Behnke in Kontakt war, 
gab es in der Bonner Zeit ab 1927 regelm\"a{\ss}ige Treffen mit Behnke und seinen Sch\"ulern. 
Gottfried K\"othe beschreibt ein Zusammentreffen von Behnke und Toeplitz 1929 in Bonn in \cite[S. 10]{bms143}: \\

\noindent {\it Als Toeplitz bei Emmy Noether anfragte, ob sie ihm einen jungen Algebraiker empfehlen k\"onnte, 
schlug sie mich vor, und ich kam im Sommer-Semester 1929 zu einem Vorstellungsbesuch nach Bonn. Ich geriet in ein gro{\ss}es Treffen 
von Toeplitz und Behnke mit ihren Sch\"ulern und machte gleich eine Rheinfahrt mit. }\\

\noindent Im Jahr 1932 gr\"undeten Behnke und Toeplitz  die Zeitschrift ,,Semesterberichte zur Pflege des Zusammenhangs von Universit\"at und Schule''.
Im Universit\"atsarchiv M\"unster befindet sich ein Geleitwort zum Erscheinen der Semesterberichte vom Juni 1932 \cite[S. 201]{hartmann2}: \\

\noindent
{\it Ich habe es immer als schmerzlich empfunden, dass seit der Aufgabe der Ferienkurse faktisch keine Beziehung mehr zwischen den Studienr\"aten der Mathematik und 
uns besteht. Das ist f\"ur den mathematischen Unterricht an Schulen und Universit\"at gleich sch\"adlich. Da andererseits infolge der weiteren Beschr\"ankung der
Mittel die Wiederaufnahme der Ferienkurse nicht m\"oglich ist, haben Herr Kollege Toeplitz aus Bonn -- der diesen Zustand ebenso stark empfindet -- und ich 
beschlossen, Semesterberichte herauszugeben. }\\
 
\noindent Die ersten Hefte der Semesterberichte begleiteten die von Behnke und Toeplitz organisierten 
,,Tagungen zur Pflege des Zusammenhangs von Universit\"at und h\"oherer Schule'', die ab 1931 in M\"unster stattfanden \cite[S. 34]{hartmann1}.  
Heinrich Behnke beschreibt die Gr\"undung der Semesterberichte auch im Nachhinein in \cite[S. 6]{behnke2}: \\

\noindent {\it 1932 wurden mit dem Verfasser dieser Zeilen trotz der gro{\ss}en Wirtschaftskrise die alten Semesterberichte in ihrer bescheidenen Form 
der blauen hektographierten Hefte (1-14) begr\"undet. Der \"au{\ss}ere Anla{\ss} daf\"ur war unser Verlangen, mit unseren fr\"uheren Studenten
auch noch in Verbindung bleiben zu k\"onnen, wenn sie zur Schule zur\"uckgekehrt waren. Auf den Gedanken, da{\ss} ein erheblicher Teil unserer uns enger verbundenen
Studenten in der Industrie Besch\"aftigung suchen sollte oder an die Hochschulen des In- und Auslandes gehen w\"urde, kam damals noch keiner von uns. 
Das Programm der Bl\"atter war im wesentlichen das der heutigen Semesterberichte (die seit 1949 bei Vandenhoeck und Ruprecht bzw. im Springer Verlag erscheinen). Doch
wer die Reihe der 14 Hefte verfolgt, die bis zum Kriege ver\"offentlicht wurden, erkennt, da{\ss} Toeplitz noch eine weitere Linie verfolgen wollte, als die Semesterberichte 
dann wirklich eingeschlagen haben. Er hatte die Absicht, in diesen Bl\"attern sich auch \"uber die Fragen des Schulunterrichts auszulassen. Doch beklagte er 
dann, da{\ss} offenbar Fachgenossen und Schulleute wenig darauf eingingen.} \\

\noindent Diese Textpassage best\"atigt die gleichberechtigte Rolle von Toeplitz bei den Semesterberichten. Toeplitz wurde ab 1935 aus der Redaktion durch Einfluss von 
Georg Hamel und Theodor Vahlen entfernt \cite[S. 193]{hartmann1}. Im Jahr 1939 wurde die Zeitschrift eingestellt. Die Gr\"unde daf\"ur waren zum einen 
finanzieller Art, aber es gibt auch Stimmen, die davon ausgehen, dass Ludwig Bieberbach die Reduktion der Fachzeitschriften in der DMV so vorangetrieben hatte, 
dass die Semesterberichte politisch unter Druck gerieten \cite[S. 209]{hartmann1}, \cite[S. 19]{remmert}. Erst nach dem Krieg wurden die Semesterberichte 1949 
unter dem neuen Namen ,,Mathematisch-Physikalische Semesterberichte'' bei Vandenhoeck und Ruprecht 
wieder von Behnke zusammen mit Wilhelm S\"u{\ss} und Walter Lietzmann herausgegeben. Sie existieren bis zum heutigen Tage, derzeit unter dem Namen 
,,Mathematische Semesterberichte'' beim Springer Verlag. \\

\noindent \"Uber das pers\"onliche Verh\"altnis zwischen Behnke und Toeplitz wissen wir nicht sehr viel, ebensowenig, ob es sich nach 1933 ver\"andert hat. 
Die erste Frau von Heinrich Behnke, Aenne Albersheim, war j\"udischer Abstammung und verstarb 1927 nach Geburt des Sohnes Hans. 
Insofern war Behnke sicherlich beim Thema Antisemitismus sensibel. Uri Toeplitz geht darauf in \cite[S. 133]{uri} ein: \\

\noindent {\it Behnke war ein h\"aufiger Gast bei uns in Bonn, immer gute Laune mit sich bringend, wie dies auch mit seiner imposanten Figur 
und Stentorstimme zusammenging. Er kam weiter zur Nazi-Zeit, obwohl er, der aus erster Ehe einen halb-j\"udischen Sohn hatte, wohl besonders vorsichtig sein musste. }\\

\noindent In Behnkes Autobiographie findet sich davon wenig. Es ist hervorzuheben, dass Behnke bis zu Toeplitz 
Abreise nach Jerusalem, ebenso wie Gottfried K\"othe, mit Toeplitz pers\"onlichen und brieflichen Kontakt hatte.
Behnke hat offenbar ab 1933 dennoch eine im Grunde \"angstliche politische Einstellung gehabt. 
Dies wird durch Peter Thullen in seinen Tageb\"uchern hervorgehoben \cite[S. 48]{thullen}, siehe auch \cite[S. 42]{hartmann1}. Volker Remmert erw\"ahnt
in \cite[S. 12]{remmert} Aktennotizen des M\"unsteraner Rektorats, die belegen, dass Behnke zwar unter Beobachtung stand, aber aufgrund seiner Verdienste einen Schutz genoss.
Es gibt einen Hinweis auf Verstimmungen zwischen Behnke und Toeplitz, den Uta Hartmann in Ihrer Dissertation \cite[S. 37]{hartmann1} aufgebracht hat. 
In einem Brief an Karl Jaspers vom 19. Oktober 1936 schreibt Toeplitz \"uber einen Vorfall, der jedoch im Unklaren bleibt:\\

\noindent {\it Es handelte sich darum, Behnke in einem kritischen Augenblick zu erziehen, wo er es wieder mal n\"otig hatte. Schwer genug in meiner Lage, aber ich wollte es schon. 
H\"atte er an Ihnen wie an allen anderen keine St\"utze gefunden, so h\"atte ich sicher mehr erreicht. Ich sehe eine Menge Arier in \"ahnlicher Lage, und B.[ehnke] mit seiner
ungemeinen und biegsamen Intelligenz m\"usste es schaffen, ebenso klar dazustehen, wie manche von diesen, die ich im Auge habe. Aber 
es r\"achen sich jetzt bei ihm einige gro{\ss}e Fehler, gegen die ich immer angek\"ampft habe ... 
Mir ist eine solche Beziehung von geringer Innerlichkeit eigentlich trauriger, als gar keine.} \\

\noindent Es ist die Frage, und bleibt offen, wie tief diese Verstimmung bei Toeplitz war, zumal die beiden Pers\"onlichkeiten sehr unterschiedlich waren. 
Umgekehrt finden in den noch existierenden Briefen von Behnke an Jaspers solche Stimmungslagen keine Erw\"ahnung \cite[S. 38]{hartmann1}.


\section*{Toeplitz und die Mathematikgeschichte} 

Bereits in Kiel hatte Toeplitz intensive Kontakte zu Heinrich Scholz und Julius Stenzel und veranstaltete mit ihnen gemeinsam 
ein mathematisch-historisches Seminar, in dem vor allem antike Mathematik besprochen wurde. 
Stenzel war ein angesehener Graezist, Scholz ein Theologe und Philosoph, der sich sp\"ater zum Logiker wandelte und 
eine Schule in M\"unster begr\"undete. In den Bonner Jahren baute Toeplitz zusammen mit
Erich Bessel-Hagen eine mathematik-historische Bibliothek auf \cite[S. 30]{bms319} und hielt auch dort ein historisches Seminar ab. 
Zu diesem Bonner Kreis um Toeplitz geh\"orte sp\"ater auch Oskar Becker, der 1931 nach Bonn berufen wurde, seine Eudoxosstudien 
dort entwickelte und in den Quellen und Studien \cite{quellen} publizierte. 
In den Semesterberichten hat Harald Boehme \"uber Oskar Becker k\"urzlich einen Artikel verfasst \cite{boehme}.  
Die Staatsexamensarbeit von Daniel Schneider \cite{schneider} behandelt ebenfalls verschiedene Aspekte dieser Person. 
Oskar Becker entwickelte in den drei{\ss}iger Jahren eine deutliche nationalsozialistische Gesinnung, die daf\"ur sorgte, dass er nach dem Krieg
von den Besatzungsm\"achten f\"ur einige Jahre in den einstweiligen Ruhestand versetzt wurde. Wie diese Ideologie mit seiner vorangegangenen 
Zusammenarbeit mit Toeplitz vereinbar ist, bleibt uns heute unverst\"andlich. \\
 
\noindent Es ist erstaunlich, dass an einigen Instituten in dieser Zeit Mathematiker gleichzeitig intensive historische Interessen 
entwickelten und Seminare dazu abhielten. 
Der Krieg hat diese Entwicklung aber wieder beendet. Es w\"are m\"oglicherweise interessant, die Ergebnisse dieser Seminare noch mehr aufzubereiten. 
Bei Toeplitz wurde oft darauf hingewiesen, beispielsweise von Max Born, dass seine Breslauer famili\"aren Wurzeln und seine humanistische 
Schulbildung die Grundlage f\"ur diese Neigung bildeten. Jedoch war sein Interesse in dieser Hinsicht weit \"uberdurchschnittlich. 
Der bemerkenswerte Kontakt zu den Philologen wurde von Toeplitz selbst als Arbeitsprogramm thematisiert \cite[S. 4]{toeplitz13}: \\

\noindent {\it Die Forschung ist hier vor den Toren des Baues, als den wir die griechische Mathematik vorstellen wollen, stehen geblieben. 
Nur ein neues System der Zusammenarbeit von Philologe und Mathematiker kann diese Tore \"offnen.} \\

\noindent Toeplitz hatte also nicht nur ein ausgepr\"agtes Interesse an mathematisch-historischen Fragen, sondern auch ein Bed\"urfnis, mit ad\"aquaten f\"acher\"ubergreifenden 
Methoden zu arbeiten. Julius Stenzel wurde 1925 nach Kiel berufen. Er hatte mit seinem Buch \cite{stenzel} zu Plato und Aristoteles die 
Arbeiten von Toeplitz \cite{toeplitz12,toeplitz13,toeplitz14} zur griechischen Mathematik beeinflusst. In diesen Artikeln begab sich Toeplitz intensiv in den Diskurs, der 
zu diesen Themen in seiner Zeit herrschte, zum Beispiel durch Oswald Spenglers ,,Der Untergang des Abendlandes'', und bewies oder widerrief Thesen, indem er genaue 
\"Ubersetzungen anforderte, um den Interpretionsspielraum kritisch zu kontrollieren. So sollten zum Beispiel Fehlinterpretationen durch Analogien mit moderneren mathematischen 
Begriffen ausgeschlossen werden. Durch diese Vorgehensweise kam er zu einer Kritik an Thesen von A. E. Taylor. In eine \"ahnliche Auseinandersetzung begab sich Toeplitz 
mit seiner Widerlegung der Kritik von Ernst Mach an der Darstellung des Archimedischen Beweises des Hebelgesetzes \cite{mach}. 
Zwei seiner Doktoranden, Walter Stein \cite{stein} und Dora Reimann \cite{reimann}, publizierten \"uber Toeplitz Kritik an Mach
in den Quellen und Studien, was einer massiven Kritik an Mach gleichkam. In der Staatsexamensarbeit von Sebastian Lange \cite{lange} wird dies ausf\"uhrlich behandelt. \\
  
\noindent Die Breite von Toeplitz Denken wird auch offenbar durch die engen Kontakte, die Toeplitz in Bonn mit Professorenkollegen aus anderen Gebieten pflegte: 
da war Hans Cloos (Geologe), Richard Thoma (Jurist), Paul Ernst Kahle (Orientalist), Leo Waibel (Geograph), Fritz Kern (Historiker) und Adolf Zycha (Rechtshistoriker) \cite[S. 28]{bms319}. 
In letzter Konsequenz hatte Toeplitz zusammen mit Julius Stenzel und Otto Neugebauer 1929 die Zeitschrift ,,Quellen und Studien zur Mathematik, Astronomie und Physik'' gegr\"undet, die
von 1931 bis 1938 beim Springer Verlag erschien. Hier ein Zitat aus dem Geleitwort zur ersten Ausgabe, das wieder die \"ubergreifenden Ideen von Toeplitz in Bezug auf Historie und Didaktik 
beleuchtet \cite[Band 1, S. 1]{quellen}: \\ 

\noindent {\it Die heute so aktuellen Bem\"uhungen um die Grundlagen der Mathematik, das damit eng zusammenh\"angende Interesse an philosophischen und didaktischen 
Problemen haben mit gutem Recht auch die Frage nach dem geschichtlichen Werdegang mehr in den Vordergrund geschoben. Wir glauben daher, den Versuch wagen zu d\"urfen, 
der Forschungen
nach den geschichtlichen Grundlagen der mathematischen Wissenschaften eine neue St\"atte zu bieten. }\\

\noindent  In den Quellen und Studien erscheinen Artikel von Erich Bessel-Hagen, Oskar Becker und Bartel Leendert van der Waarden. 
Die Zeitschrift ist die zweite, die Toeplitz gr\"undete, und bei der er seinen hohen Anspruch an die Mathematikgeschichte und Philologie 
einflussreich etablierte. Der Mitherausgeber Otto Neugebauer emigrierte 1933 nach Kopenhagen und 
Julius Stenzel starb 1935 in Halle. Die Zeitschrift lief daher aus, da Toeplitz ohnehin als 
j\"udischer Herausgeber Ablehnung erfuhr. Oskar Becker machte sich im Dezember 1938 in einem Brief an Otto Neugebauer sogar f\"ur die Absetzung von Toeplitz als Herausgeber 
stark, worauf Neugebauer aus Kopenhagen eine scharfe Antwort schrieb \cite[S. 142]{RSS} und seinen eigenen R\"ucktritt ank\"undigte: \\

\noindent {\it Sie schrieben mir, da{\ss} Sie als Nationalsozialist anscheinend eine von der meinen verschiedene Auffassung haben, trotz ihrer pers\"onlichen
Hochachtung vor Herrn Toeplitz. Ich kann darauf nur antworten, da{\ss} ich nicht im gl\"ucklichen Besitz irgendeiner ''Weltanschauung'' bin und 
daher gen\"otigt bin, mir in jedem einzelnen Fall zu \"uberlegen, was ich tun soll, ohne mich auf ein von vorneherein gegebenes Dogma zur\"uckziehen zu k\"onnen.
Diesen Nachteil dem praktischen Leben gegen\"uber wiegt vielleicht auf, da{\ss} es mir erspart bleibt, mich von den Menschen, die ich hochsch\"atze, deshalb
trennen zu m\"ussen, weil sie das Ungl\"uck haben, von anderen Menschen mi{\ss}handelt zu werden.} \\

\noindent In diesen Zeilen eines bewundernswerten Charakters steckt nicht nur die Verteidigung von Otto Toeplitz, sondern auch die kaum versteckte
Verachtung der Gedankenwelt dieses Oskar Becker, der durch seine Forderung vom Mitl\"aufer zum T\"ater wurde.

\section*{Die genetische Methode}

In einem Vortrag am 24. September 1926 vor dem Reichsverband \cite{toeplitz4} legte Otto Toeplitz in gro{\ss}er Klarheit die 
hochschuldidaktischen Probleme in der Anf\"angervorlesung ,,Infinitesimalrechnung'' dar. Er berief sich dabei auf 
eine \"ahnliche Auseinandersetzung, die auf der Tagung der Naturforscher und \"Arzte 28 Jahre fr\"uher im Zusammenhang mit
Vortr\"agen von Alfred Pringsheim und Felix Klein stattfand \cite[S. 88]{toeplitz4}:  \\

\noindent {\it Der tats\"achliche Zustand dieser Vorlesung an den deutschen Universit\"aten zeigt noch heute die gleiche bunte Mannigfaltigkeit; auf der einen Seite
die strenge Observanz, die mit einer sechsw\"ochentlichen Dedekindkur anhebt und dann aus den Eigenschaften des allgemeinen Zahl- und Funktionsbegriffs 
die konkreten Regeln des Differenzierens und Integrierens herleitet, als w\"aren sie notwendige, nat\"urliche Konsequenzen, auf der anderen Seite die 
anschauliche Richtung, die den Zauber der Differentiale walten l\"a{\ss}t und auch in der letzten Stunde der zwei Semester umspannenden Vorlesung den Nebel, 
der aus den Indivisibilien aufsteigt, nicht durch den Sonnenschein eines klaren Grenzbegriffs zerrei{\ss}t; und dazwischen die hundert 
Schattierungen von Diagonalen, die man zwischen zwei zueinander senkrechten Ideenrichtungen einzuschalten vermag.}  \\

\noindent Toeplitz versucht in \cite{toeplitz4}, dieses Problem zu erkl\"aren. Er f\"uhrt dazu drei sogenannte Momente an: Das erste Moment 
ist die Kluft zwischen Exaktheit und Anschaulichkeit, die besonders in der Mathematik steckt, seit der Anspruch der Exaktheit dort so wichtig geworden ist. 
Toeplitz f\"uhrt Weierstra{\ss} als Vorreiter der mathematischen Strenge an. Das zweite Moment ist in der Rolle der Anf\"angervorlesung 
zwischen Schule und weiterf\"uhrenden Kursvorlesungen begr\"undet. Dies zeigt sich in der weiten Spanne zwischen praktischen Rechen- und Beweistechniken und 
z.B. der Epsilontik, d.h. dem strengen Grenzwertbegriff. Dieses inhaltliche Problem wird \"uberlagert von der Inhomogenit\"at an Begabung 
der H\"orer einer solchen Vorlesung. Er trifft die Entscheidung, dass die 45 \% der mittelguten H\"orer (ohne die 5 \% der absolut Begabten und die 50 \% der Unbegabten)
im Mittelpunkt des Interesses stehen m\"ussen. Das dritte Moment ist die \"Uberschneidung mit der Schule, heute Schnittstelle genannt. Die Art und Weise wie in der Schule
Differentialrechnung gelehrt werde, sei auf das Gros der Sch\"uler abgestellt, so dass die Schule das Problem nicht unbedingt verbessere. Toeplitz schl\"agt in der Folge
eine L\"osung dieser Probleme dar, die sich die ,,genetische Methode'' nennt. Diese Methode stellt auf eine motivierende Herangehensweise an Begriffe ab, \"ahnlich der 
historischen Entstehungsgeschichte. Er differenziert aber klar von einer ,,historischen Methode'', so dass es hier nicht darum geht, Begriffe in ihrer
historischen Entwicklung zu beschreiben, sondern das Herangehen in der Entstehung (Genese) zu verbessern. Er sieht das nicht als neue Idee, 
sondern erinnert wieder an die Diskussion zwischen Klein und Pringsheim, wo auf das ,,biogenetische Grundgesetz'' hingewiesen wurde. 
Es folgt ein Beispiel zur genetischen Methode 
an Hand der Reihenkonvergenz, bei dem historische Beispiele von Reihen (Mercator, Newton, Euler, Bernoullis) einer Konvergenzbetrachtung vorangestellt werden. 
Otto Toeplitz hat in Vorlesungen und mit Hilfe seiner Arbeitsgruppe an einem tats\"achlichen Umsetzungskonzept gearbeitet, das posthum von Gottfried K\"othe 
publiziert wurde \cite{toeplitz1}.  Die Folgen aus den drei Momenten werden schlie{\ss}lich in drei Thesen zusammengefasst: Zum ersten sei die Integralrechnung vor der Differentialrechnung 
zu behandeln, was durch die Griechen (Archimedes) und Dedekind, Fermat und Cavalieri untermauert wird. Die zweiten These stellt den Hauptsatz der Differential- und
Integralrechnung in den Mittelpunkt. Anhand der Personen Fermat, Barrow, Newton und Leibniz wird der Unterschied zwischen Genesis und Historie 
f\"ur Toeplitz nochmal deutlich, und man sp\"urt seine Gelehrtheit, wenn er den Priorit\"atsstreit in die historische Ecke stellt, 
und die Genesis betont. 
Seine dritte These ist die Behauptung, dass Grundbegriffe (wie der Zahlbegriff) in das zweite Semester geh\"oren. 
Als Zusatz bemerkt er, dass der Cantorsche Zahlbegriff dem von Dedekind vorzuziehen sei, nachdem er eine Fundamentalkritik an den Dedekindschen Schnitten
vorgenommen hat. Wenngleich die Idee der Schnitte im Wesentlichen schon in der griechischen Mathematik auftaucht, wie Toeplitz in \cite[S. 97]{toeplitz4} anmerkt, 
lehnt er diese in der Anf\"angervorlesung aus wohlbekannten didaktischen Gr\"unden ab. 
Die Sch\"onheit der Idee von Dedekind, und die M\"oglichkeit eines genetischen Zugangs zu ihr ignoriert er, weil
er Dedekind Ungenauigkeiten unterstellte, die aus heutiger Sicht behebbar sind. 
Als vierte These will Toeplitz die Differentialrechnung nur noch bei ,,rechnerischen'' Funktionen im Gegensatz 
zu ,,willk\"urlichen'' Funktionen sehen. Dies m\"undet in einen historischen Exkurs in die Welt von Barrow, Leibniz und Newton, mit der These, dass die 
Neuerungen von Leibniz und Newton sich nur auf die Weiterentwicklung zum Kalk\"ul, also die Berechenbarkeit beziehen. Diese These zeigt wieder 
die historischen Interessen von Toeplitz. In dem eigentlichen Buch \cite{toeplitz1} wird zu Anfang eine schrittweise Einf\"uhrung in die Gedankenwelt gegeben. 
Erst in \S 9 wird der Grenzwertbegriff exakt eingef\"uhrt und viele der \"ublichen Beweise aus der Epsilontik behandelt. Auf diesem Grenzwertbegriff f\"ur Folgen 
baut dann der Begriff der Reihe und des Integral auf. Das Ende des Buches bringt eine Diskussion der Keplerschen Gesetze. \"Uberhaupt enth\"alt der Stoff 
viele physikalische Anwendungen, \"Ubungsaufgaben, die Toeplitz so wichtig waren, sowie viele historische Bemerkungen und motivierende Abschnitte.
Das bemerkenswerte Buch ist wom\"oglich vielen Analysisb\"uchern der heutigen Zeit \"uberlegen. Die genetische Methode hat sich seitdem verselbstst\"andigt und
wird in vielen Kontexten, neben der historisch-genetischen auch in einer psychologisch-genetischen Variante erw\"ahnt \cite{schubring}. Toeplitz Zugang bezieht 
sich jedoch eindeutig auf die historisch-genetische Richtung, wenngleich er auch die Abgrenzung zum rein historischen Zugang w\"unschte. 
Der genetische Zugang von Toeplitz hat nat\"urlich auch seine Nachteile. Heinrich Behnke macht dazu in \cite[S. 4]{behnke2} eine Bemerkung: \\

\noindent {\it Die einzelnen Fundamentals\"atze sollen historisch entwickelt und dadurch dem H\"orer in einer nat\"urlichen Spannung und zugleich p\"adagogisch naheliegenden
Weise vorgetragen werden. Es liegt nahe zu fragen, wie dazu die erforderliche Zeit gewonnen werden kann.} \\

\noindent Der Schluss des Artikels von Toeplitz wendet sich gegen die Suche nach der absoluten Objektivierung in der Mathematik. 
Toeplitz reiht Hilbert und Weyl in sein Denkschema ein. 
Man merkt bereits in diesem Aufsatz, wie sich Toeplitz von seinen Vorbildern Klein und Hilbert abzugrenzen versucht.

\section*{Die Schnittstelle zwischen Schule und Universit\"at}

Das Problem der Schulmathematik behandelte Toeplitz bereits in seiner Rede \cite{toeplitz4} von 1926. Zwei Jahre sp\"ater hielt er einen Vortrag auf der Tagung der
Naturforscher und \"Arzte \cite{toeplitz5}, der sich genauer mit der Schnittstelle Schule-Hochschule besch\"aftigte. Als Hintergrund muss man wissen, dass 
sich das Schulcurriculum durch den Einfluss besonders von Felix Klein in diesen Zeiten ge\"andert hatte (Kleinsche Reform). Etwa im Jahr 1925 wurde die Differentialrechnung 
als Schulstoff etabliert. Zu Anfang der Rede konstatierte Toeplitz die verfahrene Situation mit den Worten: \\

\noindent {\it Das ist das Bild eines v\"olligen Auseinanderlebens der beiden Instanzen in allen ihren Organen.}\\ 

\noindent Er ging dabei von der damaligen Zahl von 95 \% an Studierenden aus, die einmal in den Schuldienst gehen w\"urden. 
Bis auf diese Beschreibung \"ahnelt die Sachlage unserer heutigen Zeit. Toeplitz konzentriert sich dann auf die Hochschuldidaktik und 
bringt sein Anliegen auf den Punkt, nachdem er von den Schwierigkeiten in Kleins Vorlesung ,,Elementarmathematik vom h\"oheren Standpunkte aus'' berichtet hat 
\cite[S. 3]{toeplitz5}: \\

\noindent {\it Felix Klein hat eine Vorlesung \"uber Elementarmathematik vom h\"oheren Standpunkt 
geschaffen, die zweifellos eine gro{\ss}e Wirkung ausge\"ubt hat. Indessen
hat dieser Versuch wenig Nachahmung durch andere Dozenten gefunden, und
mit Recht ... so erscheint hier eine Materie von
dem \"ubrigen Stoff der Universit\"atsmathematik losgel\"ost und f\"ur sich behandelt: 
die Elementarmathematik und was stofflich mit ihr verwandt ist ... Und gerade ein
solcher Versuch kann nicht in die Tiefe wirken. Aus dem methodischen Rahmen
der Algebra herausgerissen ... 
Der Ausgleich zwischen Stoff und Methode in den heutigen Universit\"atsvorlesungen ist im Prinzip nicht der richtige. In diesem Ausgleich
sehe ich den Schl\"ussel zur L\"osung des ganzen Aufgabenkomplexes, der von der Leitung des Kongresses mit dem Wort von der ,,Spannung'' so trefflich
gekennzeichnet worden ist. } \\

\noindent In diesen Bemerkungen zu Felix Klein wird wieder deutlich, wie Toeplitz im Begriff ist, sich von seinem Vorbild zu l\"osen. 
Neben der Kritik ist dies besonders in seinem viel weitergehenden inhaltlichen Ansatz zu sp\"uren. Man kann argumentieren, dass Felix Klein 
einen Standpunkt hatte, der aufgrund seiner Rolle in erster Linie wissenschaftspolitisch angelegt war und Toeplitz sich bem\"uhte, die Praxis der Ausbildung 
mit Leben zu f\"ullen. Die Verschiedenheit der Personen Toeplitz und Klein wird hier jedenfalls deutlich. \\

\noindent Man muss wissen, dass es zur Zeit dieser Rede an manchen Universit\"aten eine gro{\ss}e Anzahl von H\"orern gab. 
Behnke gibt in \cite[S. 5]{behnke2} an, dass es in Bonn 1928-1930 mehrere hundert Mathematikstudenten gab, und die Anf\"angervorlesung 200 H\"orer hatte.
Dies nimmt Toeplitz zum Anlass, im Abschnitt A auf die diesbez\"uglichen Probleme einzugehen, indem er intensiv das Pr\"ufungswesen bespricht. Interessanter sind seine Bemerkungen 
zu seinem \"Ubungsbetrieb, der ihm sehr am Herzen lag und der unserem heutigen \"ahnlich gewesen sein muss. 
Im Hinblick auf die Abbrecherquote ist ihm das Eingehen auf den einzelnen Studierenden sehr wichtig, und er 
erkl\"art im Detail seine Bonner Versuche, das Individuum zu f\"ordern. Die intensive Betreuung der Studenten war Toeplitz also sehr wichtig, und dies war Teil seines
didaktischen Konzepts. Heinrich Behnke geht auf diese Aktivit\"aten in \cite[S. 5]{behnke2} ein: \\

\noindent {\it In der allgemeinen akademischen Verantwortungs- und Hilflosigkeit f\"uhlt er die Verantwortung f\"ur die allzu gro{\ss}e Zahl seiner H\"orer. 
Das m\"u{\ss}te einmal zum Verderben der Schulen und ihrer Lehrer ausschlagen, wenn man nicht die weniger geeigneten zeitig aus den H\"ors\"alen entfernte. Mit gr\"o{\ss}ter
Umsicht machte er sich sogleich an die Aufgabe, wenigstens in seinem Felde den schw\"acheren Studenten zeitig und in einer warmherzig sorgenden Art den Weg zu versperren. Nie
las er Vorlesungen ohne \"Ubungen. Dort achtete er genau auf die Leistungen der Teilnehmer. Er kannte die geistigen F\"ahigkeiten seiner Studenten im einzelnen. So konnte
bei ihm nicht leicht eine Arbeit unter fremden Namen laufen, weil er in Ansatz, Ausdruck und Schrift die Eigenarten seiner 200 H\"orer kannte und sie immer wieder im Kreise
seiner Mitarbeiter besprach. Die Begabtesten aber aus der gro{\ss}en Schar sammelte er zu vielen Sonderbesprechungen um sich. In v\"ollig ungezwungenem und kameradschaftlichen Umgange,
der auch auf ihn durch viele Anregungen zur\"uckstrahlte, bem\"uhte er sich, ihre fachliche und allgemeine geistige Entwicklung wesentlich zu f\"ordern. Wenn man ihn aufsuchte, fand man 
ihn meistens in diesem geistig und menschlich bewegten Kreise, dem er ebensosehr sorgender Freund wie Lehrer war.}\\

\noindent In diesem Zusammenhang ist auch der Bericht von Gottfried K\"othe aus \cite[S. 10]{bms143} interessant: \\

\noindent {\it  Man traf sich nachmittags bei Toeplitz in seiner sch\"onen Wohnung in der Coblentzer Stra{\ss}e und korrigierte gemeinsam. Wir besprachen anhand der ersten 10 
Hefte die m\"oglichen L\"osungsmethoden; wir wu{\ss}ten, welche der Studenten zu originellen Gedanken f\"ahig waren und kamen so recht bald zu einer Einteilung, in die die gro{\ss}e Masse
der \"Ubungen dann leichter einzuordnen war. Toeplitz legte gro{\ss}en Wert auf sehr sorgf\"altige Durchsicht. Wir waren zu dritt oder zu viert, ich erinnere mich vor allem an Frau Hagemann
und an Helmut Ulm, der gerade an seiner sp\"ater ber\"uhmt gewordenen Dissertation zu arbeiten begann. Da es meist \"uber 100 Hefte waren, sa{\ss}en wir oft bis 9 Uhr abends beisammen, 
zwischendurch brachte Frau Toeplitz Berge von belegten Br\"otchen herein.  } \\

\noindent In Teil B seiner Rede geht Toeplitz auf die Infinitesimalrechnung in der Schule und ihre didaktischen Probleme ein. Er schreibt \cite[S. 15]{toeplitz5}: \\

\noindent {\it Die ganze Schwierigkeit mit der Infinitesimalrechnung auf der Schule ist dadurch entstanden, da{\ss} man sie eingef\"uhrt hat, ehe man das didaktische Problem 
gel\"ost oder auch nur ernsthaft angegriffen hatte, das eben hier aufgeworfen worden ist. Es darf nicht verheimlicht werden, da{\ss} es zur Zeit in der Hauptsache noch ungel\"ost ist.
Davon, ob es gelingt, es zu l\"osen, davon, ob man es \"uberhaupt mit voller Kraft vornimmt, wird es abh\"angen, ob die Infinitesimalrechnung auf der Schule die Stelle sich f\"ur
immer erobert, die sie soeben zu besetzen begonnen hat. Gelingt die L\"osung nicht, so wird die Infinitesimalrechnung in zwei Dezennien ebenso unr\"uhmlich von der Schule verschwinden, 
wie heute die Dreiecksaufgaben verschwunden sind.} \\

\noindent Die Rede endet mit einem Appell, das Methodische zu \"andern, nicht nur das Stoffliche.
Mir scheint, dass die damaligen Probleme der Schnittstelle bis heute ungel\"ost sind, auch wenn es dazu bis heute viele
Ans\"atze gab, womit sich auch die Semesterberichte immer wieder befasst haben.

\section*{Das mathematische Werk von Toeplitz}

Zum mathematischen \OE{}uvre von Otto Toeplitz gibt es \"Ubersichtsartikel von Peter Lax in \cite{bms319}, Jean Dieudonn\'e in \cite{gohberg}
und Gottfried K\"othe \cite{gohberg,koethe}. 
Toeplitz hat bei David Hilbert den H\"ohepunkt von Hilberts Arbeiten \"uber Integralgleichungen miterlebt. Sp\"ater hat er zusammen mit Ernst Hellinger 
einen Beitrag f\"ur die Enzyklop\"adie der mathematischen Wissenschaften geschreiben, in dem die Hauptergebnisse dieser Theorie zusammengefasst sind \cite{toeplitz3}. Er
selbst konzentrierte sich dabei auf die algebraischen Aspekte der Theorie. Toeplitz hat in den Jahren danach
versucht, viele endlich-dimensionale algebraische Eigenschaften, wie das \"Ahnlichkeits- und das \"Aquivalenzproblem von Matrizen, 
auf den unendlich-dimensionalen Fall zu verallgemeinern. Zusammen mit K\"othe hat er die Theorie der vollkommenen R\"aume 
entwickelt, die schlie{\ss}lich in der Theorie der lokalkonvexen und normalen R\"aume aufgegangen sind. 
Seine Verhaftung in algebraischen Ausdrucksweisen kam aus der Tradition von Hilbert, den er - wie wir oben beschrieben haben - als ,,Algebraiker'' bezeichnet 
hat \cite{toeplitz7}, und kann vielleicht auch durch seine fr\"uhen Besch\"aftigungen mit der algebraischen Geometrie erkl\"art werden. 
Man kann sagen, dass die Entwicklung der abstrakten Funktionalanalysis Toeplitz irgendwann \"uberholt hatte. \\

\noindent Wir wollen nun einige mathematische Resultate von Toeplitz erkl\"aren, mit denen sein Name eng verbunden ist. 
Ein bekannter, elementarer Satz von Toeplitz ist 

\begin{theorem}[Hellinger-Toeplitz] 
Sei $T$ ein selbstadjungierter (symmetrischer) linearer Operator in einem Hilbertraum $H$, dann ist $T$ stetig.
\end{theorem}

\begin{proof} Sei $\Gamma_T$ der Graph von $T$. Wegen des Satzes von der Abgeschlossenheit des Graphen ist zu zeigen, dass jede konvergente Folge 
$(f_n,T(f_n))$ in $\Gamma_T$ auch ihren Limes in $\Gamma_T $ hat. Da $T$ linear ist, k\"onnen wir annehmen, dass $(f_n)$ eine Nullfolge in $H$ ist.
Sei $g:= \lim \limits_{n \rightarrow \infty} T(f_n)$. Wir m\"ussen zeigen, dass $g=0$ gilt. Dazu benutzen wir die Stetigkeit des Skalarproduktes:   
$$
\langle g,g \rangle = \langle \lim \limits_{n \rightarrow \infty}  T(f_n),g \rangle =  \lim \limits_{n \rightarrow \infty} \langle T(f_n),g \rangle 
=\lim \limits_{n \rightarrow \infty} \langle f_n, T(g) \rangle = 
$$
$$
 = \langle \lim \limits_{n \rightarrow \infty}  f_n, T(g) \rangle = \langle 0,  T(g) \rangle =0. \hskip2cm  \qed
$$

\end{proof}

\noindent Die Bedeutung des folgenden sogenannten Permanentsatzes von Toeplitz kommt aus der Theorie der Summierbarkeit 
von konvergenten und divergenten Reihen durch Mittelbildungsverfahren, bei denen man die Folge der Partialsummen $s_n=\sum_{k=0}^n a_k$ einer 
unendlichen Reihe $\sum_{k=0}^\infty a_k$ durch die Folge der Ausdr\"ucke 
$$
t_n=\sum_{q=0}^{\infty} A_{nq} s_q 
$$
ersetzt, wobei f\"ur jedes $n$ nur endlich viele Koeffizienten $A_{nq}$ von Null verschieden sind. 
Ces\`aro Summation ist ein bekanntes Beispiel daf\"ur, bei dem 
$$
t_n=\frac{s_0+ \cdots + s_n}{n+1}
$$ 
ist. Der Satz von Lip\'ot Fej\'er, der besagt, dass die Ces\`aro Summation einer Fourierreihe immer gleichm\"a{\ss}ig gegen die Funktion $f$ konvergiert, falls
$f$ stetig und periodisch ist, ist eines der wichtigsten Beispiele f\"ur solche Summationen.  
Ein solches Mittelbildungsverfahren hei{\ss}t regul\"ar, wenn der Grenzwert bei konvergenten Reihen erhalten bleibt. 

\begin{theorem}[Permanentsatz von Toeplitz \cite{toeplitz16}] Eine solche unendliche Matrix 
$A=(A_{pq})$ definiert ein regul\"ares Mittelbildungsverfahren genau dann, wenn die folgenden drei Bedingungen erf\"ullt sind: 
\begin{itemize}
\item Die Zeilensummen $\sum_q A_{pq}$ erf\"ullen $\lim \limits_{p \rightarrow \infty} \sum_q A_{pq}=1$.
\item Die Spalten erf\"ullen $\lim \limits_{p \rightarrow \infty}  A_{pq}=0$ f\"ur alle $q$. 
\item Die Summe der Zeilenbetr\"age $\sum_q |A_{pq}|$ ist beschr\"ankt unabh\"angig von $p$.
\end{itemize}
\end{theorem}

\begin{proof}
Dies ist das Hauptergebnis in \cite[S. 119]{toeplitz16}.   \hfill \qed 
\end{proof}

\noindent Gemeinsam mit Felix Hausdorff bewies Toeplitz ein Theorem zum Wertebereich von Operatoren. Sei $T: H \to H$ ein linearer (nicht notwendig beschr\"ankter) Operator
in einem Hilbertraum $H$. Der numerische Wertebereich $W(T)$ von $T$ ist die Menge aller Werte
$$
\{ \langle T(x),x \rangle \colon   ||x||=1 \}. 
$$

\begin{theorem}[Satz von Toeplitz-Hausdorff \cite{toeplitz17}]
Der numerische Wertebereich $W(T)$ ist konvex. 
\end{theorem}

\begin{proof} 
Ein elementarer Beweis findet sich in \cite{gustafson}.  \hfill \qed 

\end{proof}

\noindent Wir wollen zum Abschluss erkl\"aren, auf welche Weise die sogenannten Toeplitzmatrizen entstehen, die so weit verbreitet in der Mathematik und in Anwendungen sind.  
Sei dazu $f(\varphi)$ eine reelle, periodische Funktion mit Periode $2\pi$ und 
\[
\sum_{n \in \Z} c_n e^{i n \varphi}
\]
die Fourierreihe von $f$. Toeplitz stellt sich in der Arbeit \cite{toeplitz6} die Frage, wie man die Positivit\"at von $f$ auf $[0,2\pi]$ durch 
ein Kriterium an die Koeffizienten $c_n$ ausdr\"ucken kann. Ein solches folgt leicht aus folgender Rechnung mit 
Hilfsvariablen $\xi_n$ ($n \in \Z$), welche Peter Lax in \cite[S. 87-88]{gohberg} andeutet, 
unter Verwendung der Orthonormalbasis $\{e^{i n \varphi} \mid n \in  \N_0 \}$:
\begin{align*}
\int_0^{2\pi} \left| \sum_{n=0}^N \xi_n e^{i n \varphi} \right|^2 f(\varphi) d \varphi 
= & \int_0^{2\pi} \sum_{m,n=0}^N \xi_n \bar{\xi}_m e^{i (n-m) \varphi} \sum_{k \in \Z} c_k e^{i k \varphi} d \varphi  \\
= & 2 \pi \sum_{m,n=0}^N c_{m-n} \xi_n \bar{\xi}_m. 
\end{align*} 
Den Ausdruck 
\[
\sum_{m,n=0}^\infty c_{m-n} \xi_n \bar{\xi}_m 
\]
nennt Toeplitz eine L-Form. Somit ist $f$ positiv genau dann, wenn alle quadratischen Formen $\sum_{m,n=0}^N c_{m-n} \xi_n \bar{\xi}_m$ positiv definit sind, 
d.h.  f\"ur alle $N$ deren Determinanten positiv sind. Die symmetrische Matrix zu einer solchen Form hat die Gestalt, die man Toeplitzmatrix nennt, 
\[
\left| \begin{array}{lllllll}
c_0    & c_1 & c_2 & c_3 & \ldots & \ldots & \ldots \cr
c_{-1} & c_0 & c_1 & c_2 & c_3 & \ldots & \ldots  \cr
c_{-2} & c_{-1} & c_0 & c_1 & c_2 & c_3 & \ldots  \cr
c_{-3} & c_{-2} & c_{-1} & c_0 & c_1 & c_2 & \cdots  \cr
\vdots & \ddots & \ddots & \ddots & \ddots & \ddots & \ddots
\end{array} \right|
\]
mit Eintr\"agen $c_{ij}=c_{j-i}$, die nur von $j-i$ abh\"angen, also auf den Diagonalen konstant sind.                   
Toeplitz zeigt weiterhin, dass der Operator, der durch eine (regul\"are) L-Form gegeben wird, als Spektrum 
genau die Funktionswerte der zu der L-Form zugeordneten analytischen Funktion 
\[
\sum_{n \in \Z} c_n z^n 
\]
auf dem Einheitskreis besitzt, siehe \cite[Satz 5]{toeplitz6}. \\

\noindent
{\bf Dank:} Ich danke Hans Niels Jahnke, Sebastian Lange, Martin Mattheis, Astrid Mehmel, Sieglinde M\"uller-Stach, David Rowe, Daniel Schneider, J\"orn Steuding, 
Duco van Straten, Renate Tobies und Klaus Volkert f\"ur Diskussionen, Verbesserungen und Hinweise, der Universit\"atsbibliothek Bonn f\"ur die 
Digitalisierung des Nachlasses sowie die Bonner Mathematischen Schriften 
und der Universit\"atsbibliothek G\"ottingen f\"ur die Kopien der Briefe von Toeplitz an Hilbert.

\end{document}